\documentclass[12pt]{amsart}
\usepackage{fullpage}

\theoremstyle{theorem}
\newtheorem{theorem}{Theorem}
\newtheorem{proposition}[theorem]{Proposition}
\newtheorem{lemma}[theorem]{Lemma}
\newtheorem{corollary}[theorem]{Corollary}

\theoremstyle{definition}
\newtheorem{definition}[theorem]{Definition}

\usepackage{colonequals}
\usepackage{tikz}
\usetikzlibrary{decorations.pathreplacing}
\usepackage{amsmath,amssymb,amsthm,verbatim,mathtools,mathrsfs}
\usepackage{xspace}
\usepackage{expl3} 
\usepackage{xparse,l3keys2e}

\usepackage{wordle}
\WordleSetup{
	present color=black
}

\usepackage{tabularx,vcell}
\definecolor{darkred}{rgb}{0.7,0,0}
\definecolor{multizero}{rgb}{1,1,1}
\definecolor{multione}{rgb}{0,0,0}

\definecolor{highlightcircle}{rgb}{1,1,0}

\DeclareMathOperator\exc{exc}
\DeclareMathOperator\des{des}
\DeclareMathOperator\Exc{Exc}
\DeclareMathOperator\Des{Des}

\newcolumntype{C}{>{\centering\arraybackslash}X}

\newcommand\classic{Permutation Wordle\xspace}
\newcommand\rainbow{multicolor Permutation Wordle\xspace}
\newcommand\basicstrat{\textnormal{\textsc{CircularShift}}\xspace}
\newcommand\floor[1]{\left\lfloor{#1}\right\rfloor} 

\usepackage[bookmarks]{hyperref}
\hypersetup{
	colorlinks=true,
	linkcolor=black,
	anchorcolor=black,
	citecolor=black,
	urlcolor=black,
	pdfpagemode=UseThumbs,
	pdftitle={Permutation Wordle},
	pdfsubject={Combinatorics},
	pdfauthor={Kutin, Smithline, and Vatter},
}

\begin{document}

\title{Permutation Wordle}

\author[Samuel Kutin]{Samuel~Kutin}
\address{Center for Communications Research, 805 Bunn Drive, Princeton,
NJ 08540-1966, USA}
\email{kutin@ccr-princeton.org}

\author[Lawren Smithline]{Lawren~Smithline}
\address{Center for Communications Research, 805 Bunn Drive, Princeton,
NJ 08540-1966, USA}
\email{lawren@ccr-princeton.org}

\author[Vincent Vatter]{Vincent~Vatter}
\address{Department of Mathematics, 1400 Stadium Rd, University of Floriday, Gainseville, FL 32611, USA}
\email{vatter@ufl.edu}

\begin{abstract}
  We introduce a guessing game, ``Permutation Wordle,'' in which a
  guesser attempts to recover a setter's hidden permutation of the set
  $\{1, \ldots, n\}$. In each round, the guesser submits a word over
  the alphabet $\{1, \ldots, n\}$, and, as in the game Wordle, learns
  which entries are correct. We describe a natural strategy and prove
  that it is optimal in a strong sense: for every $r$, it solves at least as
  many secrets within $r$ rounds as any possible strategy.
  The number of permutations it solves in exactly $k+1$ rounds is the
  Eulerian number $A(n,k)$.
\end{abstract}

\maketitle

%
%

\section{Introduction.}
\label{sec:intro}

Jon Scieszka~\cite{math-curse} writes,
``You can think of almost everything as a math problem.''  We present an example of
how this happens. We start with the game Wordle, turn it into a guessing game with
permutations, and find connections to a well-studied sequence:
the Eulerian numbers.

The online game Wordle, introduced by Josh
Wardle~\cite{Benveniste2022Wordle} in October 2021 and purchased by
\emph{The New York Times} the following
January~\cite{Tracy2022Wordle}, asks the guesser to discover a secret
five-letter word in at most six guesses. In each round, the guesser
chooses a five-letter English word and learns which letters are
correct (in the secret word in the same place), which are misplaced
(in the secret word in a different place), and which are wrong (not in
the secret word at all).\footnote{Wordle has well-defined rules for
handling repeated letters, but these are not relevant to this paper:
our secrets have no repeated letters.}

One can find
strategy tips for Wordle online.
Optimal play depends on knowledge of
English vocabulary, consonant and vowel patterns, common word endings,
and the like. In this paper we strip out the language and make the game more
abstract.  We replace the English lexicon with the symmetric group:
the secret is a permutation of $[n] = \{1,
2, \ldots, n\}$, and a guess may be any word of length $n$ with
letters from the alphabet $[n]$. We call the resulting game
\classic. The feedback in \classic is simpler than in original Wordle:
because the secret contains every number exactly once, we already
know which ``letters'' are in the answer.  Hence, the only useful
feedback is listing which entries are in the correct place.

What is the optimal strategy? How many guesses do we need on average?
More generally, what is the probability that we solve the secret
within $r$ guesses? In Section~\ref{sec:circular-shift} we offer a
strategy: \basicstrat. In Section~\ref{sec:eulerian}, we give background
on the \emph{Eulerian numbers}, a doubly indexed sequence that
frequently arises in the analysis of permutations~\cite{petersen}.
In Section~\ref{sec:eulerian-connect} we show
that the number of permutations \basicstrat solves in exactly $r$
guesses is the Eulerian number $A(n, r-1)$.
In Section~\ref{sec:proof}, we show that this strategy is optimal: for
any~$r$, \basicstrat solves the largest possible number of
permutations within~$r$ rounds. In particular, it minimizes the
expected number of guesses, which is $(n+1)/2$.

The appearance of Eulerian numbers is one sign that we have found a
``natural'' math problem from Wordle. Another is that this
permutation-guessing problem seems to keep arising.  We
first considered it not because of Wordle, but as an abstraction of
permutation-guessing games on
\texttt{\href{https://sporcle.com/}{sporcle.com}}. In July 2025,
Richard Stanley~\cite{Stanley2025}, motivated by a guessing game in
the app \emph{Royal Match}, posed the same question on MathOverflow,
asking for the optimal strategy and the expected number of guesses; a
comment on the question points to an earlier version of this
paper~\cite{arxiv-kutin-smithline} and
to related work of Hiveley~\cite{Hiveley2025}.

There are also connections between Wordle (or \classic) and the
classic board game Mastermind. We discuss these in
Section~\ref{sec:mastermind}. Finally, we mention some possible
extensions in Section~\ref{sec:open}.

%
%

\section{Circular shift.}
\label{sec:circular-shift}

We describe a simple strategy, \basicstrat, that solves \classic in at
most~$n$ guesses.\footnote{Li and Zhu~\cite{LiZhu2024} showed that no
strategy can do better than this in the worst case. See
Section~\ref{sec:mastermind} for more discussion.}  We first formalize
the Permutation Wordle setting.

Let $S_n$ denote the group of permutations on $[n]$, and let $\sigma
\in S_n$ be the setter's secret permutation. In round~$r$, the guesser
submits a guess $\gamma_r \in [n]^n$, which we regard as a function
from positions to values, writing $\gamma_r(i)$ for its $i$th letter.
The guesser then receives the feedback
\[
	U_r = \{i \mid \gamma_r(i) = \sigma(i)\},
\]
the set of positions where $\gamma_r$ agrees with~$\sigma$. Equivalently, $U_r$ is
the set of fixed points of the function $\sigma^{-1} \gamma_r$. 
The guesser may use the information $U_k$ for $1 \leq k \leq r$ in making the choice of
$\gamma_{r+1}$.


\subsection{Reducing to a uniformly random secret.}
\label{sec:random-secret}

In \classic, the setter commits to $\sigma$ at the start of the game,
so the secret does not change as the game
proceeds. (This is unlike worst-case analyses of Mastermind, which
often consider an adaptive setter; see Section~\ref{sec:mastermind}.)
Could the setter make the game harder by choosing $\sigma$ from
a nonuniform distribution on $S_n$, perhaps one tuned against the
guesser's strategy?

The answer is no, because the guesser has a response. After the
setter has committed to $\sigma$, the guesser can pick a uniformly
random permutation $\tau$ and, wherever their strategy would guess the
word $\gamma_r$, instead guess the word $\tau\gamma_r$ (defined by
$\tau\gamma_r(i) = \tau(\gamma_r(i))$ for all $i \in [n]$). The
feedback in round $r$ would then be the set of positions $i$ with
$\tau(\gamma_r(i)) = \sigma(i)$, that is, with $\gamma_r(i) =
\tau^{-1}(\sigma(i))$. This is precisely the feedback the original
strategy would receive against the secret $\tau^{-1}\sigma$, which is
uniformly random no matter how the setter chose $\sigma$. Thus, any
performance the guesser can achieve against a uniformly random secret
can be achieved regardless of the setter's distribution on $S_n$.

For simplicity, we assume for the rest of the paper that $\sigma$ is
chosen uniformly at random. In practice, one would precede any
strategy with a uniformly chosen $\tau$.


\subsection{Small~$n$.}
\label{sec:circular-shift-small-n}

We begin with the case $n=3$. Because we assume that~$\sigma$ is
uniformly random, all first guesses that are permutations are
equivalent, so we may as well choose the identity permutation as our
first guess $\gamma_1$.\footnote{For $n \le 4$, the reader can verify that
non-permutation guesses do not do better.  In general, Theorem~\ref{thm:opt} demonstrates
that \basicstrat, which always guesses a permutation, is optimal.}
In response to this first guess, we learn the set of fixed points of~$\sigma$.

\begin{figure}[h]
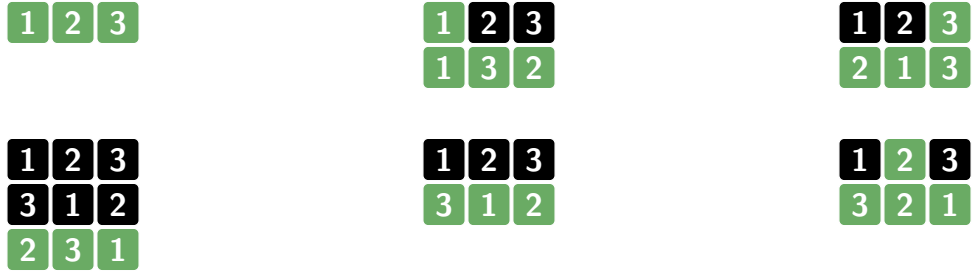

\begin{center}
\begin{tabularx}{\textwidth}{CCC}
	\savecellbox{\scalebox{0.7}{\begin{wordle}{123}
	123
	\end{wordle}}}
&
	\savecellbox{\scalebox{0.7}{\begin{wordle}{132}
	123
	132
	\end{wordle}}}
&
	\savecellbox{\scalebox{0.7}{\begin{wordle}{213}
	123
	213
	\end{wordle}}}
\\ [-\rowheight]
	\printcelltop & \printcelltop & \printcelltop
\\
& & \\
	\savecellbox{\scalebox{0.7}{\begin{wordle}{231}
	123
	312
	231
	\end{wordle}}}
&
	\savecellbox{\scalebox{0.7}{\begin{wordle}{312}
	123
	312
	\end{wordle}}}
&
	\savecellbox{\scalebox{0.7}{\begin{wordle}{321}
	123
	321
	\end{wordle}}}
\\ [-\rowheight]
	\printcelltop & \printcelltop & \printcelltop
\end{tabularx}
\end{center}
\caption{\classic when $n=3$. We solve $1$ permutation in one round, $4$ in two rounds, and $1$ in three rounds.}
\label{fig-n3}
\end{figure}

Figure~\ref{fig-n3} shows the different scenarios.
There are three possibilities, depending on the number
of fixed points of $\sigma$.
If~$\sigma$ has three fixed points, then~$\sigma$ is the identity and we are done in
one guess. If~$\sigma$
has exactly one fixed point, then~$\sigma$ is one of the three permutations that swaps two
values, and the fixed point identifies which one;
in round~$2$, we swap those two values and get the correct answer.
If~$\sigma$ has no fixed points, then~$\sigma$ is one of the two $3$-cycles,
but we do not know which. Our best strategy is
to guess one of the $3$-cycles in round~$2$; this second guess will be correct half of the time, and we will find $\sigma$ on our third guess the other half.
Overall, the probability of solving~$\sigma$ in one round is $1/6$;
in two rounds, $4/6$; in three rounds, $1/6$.

The analysis for $n=4$ is more complicated. Again, let $\gamma_1$ be the identity.
If~$\sigma$ is the identity, we finish in one round. If~$\sigma$ has one or two fixed
points, we can just keep those values the same in $\gamma_r$ for all $r > 1$; this lets
us reduce to the $n=2$ or $n=3$ case, and we finish in a total of two or three rounds.

What if~$\sigma$ is one of the nine \emph{derangements}, meaning~$\sigma$ has no fixed points? The optimal strategy is to let $\gamma_2$ be a $4$-cycle.\footnote{This choice is optimal both in the sense of maximizing the probability of solving within a particular number of rounds and in the sense of minimizing the expected number of rounds; both claims can be verified by exhaustive search, and also follow from Theorem~\ref{thm:opt}.} If $\sigma = \gamma_2$, we are done, and it took two rounds. If $\gamma_2$ agrees with $\sigma$ in one or two positions, this feedback determines~$\sigma$, and we guess it in the third round. Otherwise, the only remaining possibilities for~$\sigma$ are $\gamma_2^2$ and $\gamma_2^3$; we guess one of them in the third round, and if it is wrong, the other in the fourth. Figure~\ref{fig-n4} depicts this overall strategy.  The probabilities of finishing in one, two, three, or four
rounds are $1/24$, $11/24$, $11/24$, and $1/24$, respectively.

\begin{figure}
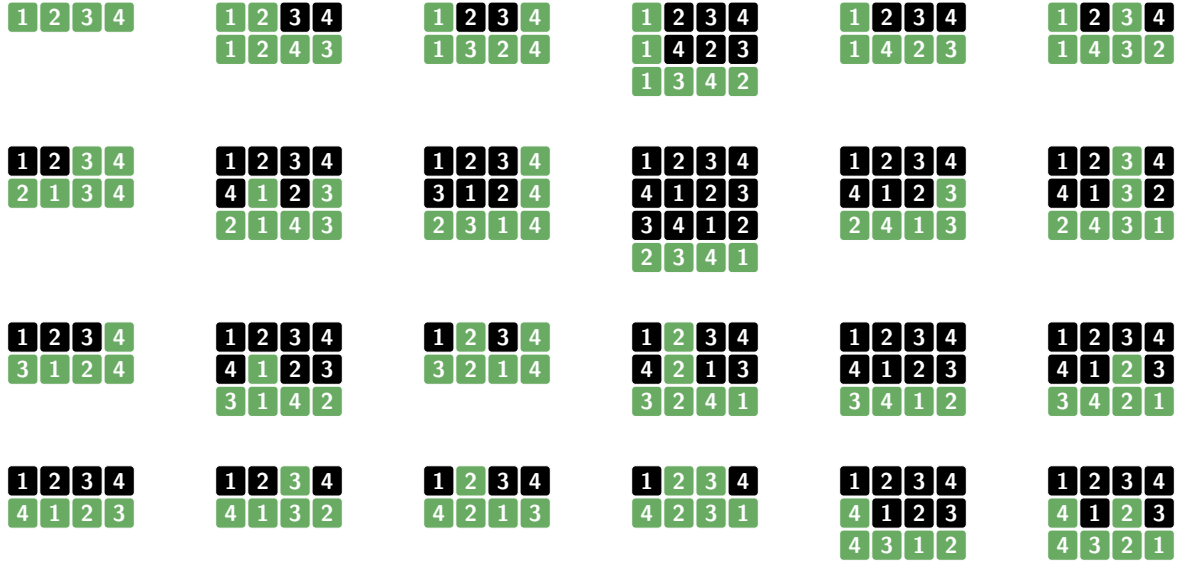

\begin{center}
\begin{tabularx}{\textwidth}{CCCCCC}
\savecellbox{\scalebox{0.5}{\begin{wordle}{1234}
  1234
\end{wordle}}} &
\savecellbox{\scalebox{0.5}{\begin{wordle}{1243}
      1234
      1243
\end{wordle}}} &
\savecellbox{\scalebox{0.5}{\begin{wordle}{1324}
      1234
      1324
\end{wordle}}} &
\savecellbox{\scalebox{0.5}{\begin{wordle}{1342}
      1234
      1423
      1342
\end{wordle}}} &
\savecellbox{\scalebox{0.5}{\begin{wordle}{1423}
      1234
      1423
\end{wordle}}} &
\savecellbox{\scalebox{0.5}{\begin{wordle}{1432}
      1234
      1432
\end{wordle}}} \\ [-\rowheight]
\printcelltop & \printcelltop & \printcelltop & \printcelltop & \printcelltop & \printcelltop\\
&&&&&\\
\savecellbox{\scalebox{0.5}{\begin{wordle}{2134}
      1234
      2134
\end{wordle}}} &
\savecellbox{\scalebox{0.5}{\begin{wordle}{2143}
      1234
      4123
      2143
\end{wordle}}} &
\savecellbox{\scalebox{0.5}{\begin{wordle}{2314}
      1234
      3124
      2314
\end{wordle}}} &
\savecellbox{\scalebox{0.5}{\begin{wordle}{2341}
      1234
      4123
      3412
      2341
\end{wordle}}} &
\savecellbox{\scalebox{0.5}{\begin{wordle}{2413}
      1234
      4123
      2413
\end{wordle}}} &
\savecellbox{\scalebox{0.5}{\begin{wordle}{2431}
      1234
      4132
      2431
\end{wordle}}} \\ [-\rowheight]
\printcelltop & \printcelltop & \printcelltop & \printcelltop & \printcelltop & \printcelltop\\
&&&&&\\
\savecellbox{\scalebox{0.5}{\begin{wordle}{3124}
      1234
      3124
\end{wordle}}} &
\savecellbox{\scalebox{0.5}{\begin{wordle}{3142}
      1234
      4123
      3142
\end{wordle}}} &
\savecellbox{\scalebox{0.5}{\begin{wordle}{3214}
      1234
      3214
\end{wordle}}} &
\savecellbox{\scalebox{0.5}{\begin{wordle}{3241}
      1234
      4213
      3241
\end{wordle}}} &
\savecellbox{\scalebox{0.5}{\begin{wordle}{3412}
      1234
      4123
      3412
\end{wordle}}} &
\savecellbox{\scalebox{0.5}{\begin{wordle}{3421}
      1234
      4123
      3421
\end{wordle}}} \\ [-\rowheight]
\printcelltop & \printcelltop & \printcelltop & \printcelltop & \printcelltop & \printcelltop\\
&&&&&\\
\savecellbox{\scalebox{0.5}{\begin{wordle}{4123}
      1234
      4123
\end{wordle}}} &
\savecellbox{\scalebox{0.5}{\begin{wordle}{4132}
      1234
      4132
\end{wordle}}} &
\savecellbox{\scalebox{0.5}{\begin{wordle}{4213}
      1234
      4213
\end{wordle}}} &
\savecellbox{\scalebox{0.5}{\begin{wordle}{4231}
      1234
      4231
\end{wordle}}} &
\savecellbox{\scalebox{0.5}{\begin{wordle}{4312}
      1234
      4123
      4312
\end{wordle}}} &
\savecellbox{\scalebox{0.5}{\begin{wordle}{4321}
      1234
      4123
      4321
\end{wordle}}} \\ [-\rowheight]
\printcelltop & \printcelltop & \printcelltop & \printcelltop & \printcelltop & \printcelltop
\end{tabularx}   
\end{center}
\caption{\classic when $n=4$. We solve $1$ permutation in one round, $11$ in two rounds, $11$ in three rounds, and $1$ in four rounds.}
\label{fig-n4}
\end{figure}


\subsection{The general algorithm.}
\label{sec:circular-shift-general}

We now extend our approach from $n=3$ and $4$ to general~$n$.  This
strategy is deterministic; the only randomness a guesser
ever needs is the relabeling described in Section~\ref{sec:random-secret}.

We call our strategy \basicstrat.  Every guess is a permutation; the
first guess is the identity permutation.  Informally, the idea is as
follows: once we get a value correct (that is, $\gamma_r(i) =
\sigma(i)$), we never change it in subsequent rounds.  We cycle the
misplaced values one step to the right to form~$\gamma_{r+1}$. Over
successive guesses, each value travels to its correct position and
stays there.

Keeping correct values in place, and requiring each guess to be a permutation,
is similar to playing Wordle in ``hard
mode.'' Hard mode requires, roughly, that every guess could be
correct, given all of the feedback received so far.
In Wordle, hard mode is a genuine restriction, and the
strongest players do better without it.\footnote{Interestingly, a
recent analysis of 730 million Wordle games shows that players who
impose the constraint on themselves do better, on average, than those
who do not~\cite{Monkovic2026WordleHard}.} An unrestricted player
can profit from a guess that abandons the revealed letters in order to
limit the remaining possibilities; for example, they can try a new letter to
see whether it appears in the secret.
Does the same ever hold in \classic?

We can give part of the answer now: in \classic,
trying a value where it is known to be incorrect never helps, since
we already know going in that the secret contains every value in $[n]$ exactly once.
To make this precise, suppose we know from a prior round that
$\sigma(i) \ne j$, but we make a guess~$\gamma$ with $\gamma(i) = j$.
Let~$\gamma'$ be identical to~$\gamma$ except that $\gamma'(i) = j'$ for
some other $j'$.
The feedback to the two guesses can differ only in position~$i$:
we know $\gamma$ is wrong there, and $\gamma'$ might be right,
so the guess $\gamma'$ might yield new information.  Also,
$\gamma'$ might be equal to $\sigma$, while
$\gamma$ cannot be. Thus, $\gamma'$ is always at least as useful a
guess as~$\gamma$. The full answer is Theorem~\ref{thm:opt}: no
strategy outperforms \basicstrat.

We state the strategy precisely.

\begin{definition}[\basicstrat]\label{def:basicstrat}
Let~$\gamma_1$ be the identity in~$S_n$. Suppose that, in round~$r$, we submit the guess $\gamma_r$ and receive the feedback $U_r$, the set of positions where $\gamma_r$ agrees with the secret. If $U_r = [n]$, we are done, and it took~$r$ rounds. Otherwise, let $[n] \setminus U_r = \{a_1 < \cdots < a_{\ell}\}$ be the set of incorrect positions, and define the next guess to be
\[
	\gamma_{r+1}(i) := \begin{cases}
	\gamma_r(i) & \text{if $i \in U_r$,} \\
	\gamma_r(a_{j-1}) & \text {if $i = a_j$ for some $j > 1$,} \\
	\gamma_r(a_\ell) & \text {if $i = a_1$.}
	\end{cases}
\]
\end{definition}

The approach depicted in Figures~\ref{fig-n3} and~\ref{fig-n4} is \basicstrat. For a larger example, see Figure~\ref{fig-n9}.

\begin{figure}
\begin{center}
	\begin{wordle}{724853169}
	123456789
	821354679
	728153469
	724853169
	\end{wordle}
\end{center}
\caption{\basicstrat guessing the secret $\sigma = 724853169$.
The strategy succeeds in four rounds.}
\label{fig-n9}
\end{figure}

Some readers may have noticed a pattern in the analysis of \basicstrat for small $n$.
For $n=3$, the number of permutations solved in $1$, $2$, or $3$ rounds are $1$, $4$, $1$; for $n=4$, the corresponding numbers are $1$, $11$, $11$, $1$.  These are the
\emph{Eulerian numbers} (see Table~\ref{tab:eulerian} below).
We discuss these numbers in Section~\ref{sec:eulerian}; we
show that they govern the performance of \basicstrat in Section~\ref{sec:eulerian-connect}.
In Section~\ref{sec:proof}, we show the Eulerian numbers 
also bound the performance
of any strategy for \classic, proving the optimality of \basicstrat.

%
%

\section{Eulerian numbers.}
\label{sec:eulerian}

The performance of \basicstrat on a permutation $\pi$ turns on the number
of \emph{excedances} of $\pi$.

\begin{definition}\label{def:exc}
For $\pi \in S_n$, an \emph{excedance} is a position $i$ for which $\pi(i) > i$.  Let
$\Exc(\pi)$ denote the set of excedances: $\Exc(\pi) := \{i \mid \pi(i) > i\}$.
Write $\exc(\pi) := \left|\Exc(\pi)\right|$, the number of excedances of $\pi$.
\end{definition}

The number of permutations in $S_n$ with a given number of excedances
is a classical quantity, the \emph{Eulerian number}:

\begin{definition}\label{def:euler}
The \emph{Eulerian number} $A(n,k)$ is $\left|\{\pi \in S_n \mid \exc(\pi) = k\}\right|$.
\end{definition}

For example, when $n = 3$, the only permutation with zero excedances
is the identity, so $A(3,0) = 1$.  The only permutation in $S_3$ with
two excedances is $231$, so $A(3,2) = 1$. We deduce that $A(3,1) = 4$.

Our proof of optimality in Section~\ref{sec:proof} will use a different characterization
of the Eulerian numbers, in terms of descents:

\begin{definition}\label{def:desc}
For $\pi \in S_n$, writing $\pi = \pi(1) \pi(2) \cdots \pi(n)$ in one-line notation, a \emph{descent} is a position $i < n$ with $\pi(i) > \pi(i+1)$. Let
\[
	\Des(\pi) = \{i \mid \pi(i) > \pi(i+1)\}, \qquad \des(\pi) = |\Des(\pi)|.
        \]
\end{definition}

It turns out that the number of permutations in $S_n$ with $k$ descents is also
given by $A(n,k)$.  One way to prove this is to show that both counts satisfy the same
recurrence.

\begin{proposition}\label{prop:excrecur}
The numbers $A(n,k)$ satisfy the recurrence
\[
	A(n,k) = (k+1) A(n-1, k) + (n-k) A(n-1, k-1)
\]
for $n \geq 1$ and $0 \leq k \leq n-1$, with initial conditions $A(0,0) = 1$ and $A(0,k) = 0$ for $k \neq 0$.
\end{proposition}

\begin{proof}
  The initial conditions hold because $S_0$ contains only the empty permutation, which has no excedances.
  
For the recurrence, we build the permutations in $S_n$ from those in $S_{n-1}$.
Given $\rho \in S_{n-1}$ and~${i \in [n]}$, we form a permutation $\pi \in S_n$:
first, append $n$ to the word $\rho(1) \rho(2) \cdots \rho(n-1)$. Second, if $i < n$, swap the
entries in positions $i$ and $n$.  Each $\pi \in S_n$ arises in exactly one way.

We compare the excedances of $\pi$ to those of $\rho$.  If $i = n$,
then no entry has moved, and $\exc(\pi) = \exc(\rho)$.  If $i < n$,
then position~$i$ is an excedance of $\pi$, while position $n$ is not.  Every other position
in $\pi$ keeps its status.  Hence, for the $1 + \exc(\rho)$ choices consisting of $i = n$
together with the $i \in \Exc(\rho)$, we have $\exc(\pi) = \exc(\rho)$; for the remaining
$n-1-\exc(\rho)$ choices, we have
$\exc(\pi) = 1 + \exc(\rho)$.  Summing over $\rho \in S_{n-1}$ proves the recurrence.
\end{proof}

\begin{proposition}\label{prop:desrecur}
The number of permutations $\pi \in S_n$ with $\des(\pi) = k$ is given by $A(n,k)$.
\end{proposition}

\begin{proof}
  Let $D(n,k)$ be the number of permutations in $S_n$ with exactly $k$ descents.  It
  suffices to show that these numbers satisfy the same recurrence and initial conditions
  as $A(n,k)$; that is,
\[
	D(n,k) = (k+1) D(n-1, k) + (n-k) D(n-1, k-1)
\]
for $n \geq 1$ and $0 \leq k \leq n-1$, with initial conditions $D(0,0) = 1$ and $D(0,k) = 0$ for $k \neq 0$.

The initial conditions hold because $S_0$ contains only the empty permutation, which has no descents.

For the recurrence, we build the permutations in $S_n$ from those in
$S_{n-1}$. Given $\rho \in S_{n-1}$ and~${i \in [n]}$, insert $n$ at
the end of the word $\rho(1) \rho(2) \cdots \rho(n-1)$ if $i = n$, and
immediately before the value $i$ if $i < n$. Every $\pi \in S_n$
arises in exactly one way; we recover $\rho$ by deleting $n$, and we
recover $i$ as the value following $n$, or as $n$ itself if $n$ is the
final entry.

We compare the descents of $\pi$ to those of $\rho$. If $i = n$, then
no adjacent pair has changed, and $\des(\pi) = \des(\rho)$.
If $i < n$, and $i$ is one of the $\des(\rho)$ values immediately following
some $j \in \Des(\rho)$, then we are replacing one descent with another,
and $\des(\pi) = \des(\rho)$.  If $i$ is one of the remaining $n-1-\des(\rho)$
choices, then $\des(\pi) = 1 + \des(\rho)$.
Summing over $\rho \in S_{n-1}$ proves the recurrence.
\end{proof}

\begin{table}
\begin{center}
\begin{tabular}{c|ccccc}
	$n \backslash k$ & $0$ & $1$ & $2$ & $3$ & $4$ \\ \hline
	$1$ & $1$ & & & & \\
	$2$ & $1$ & $1$ & & & \\
	$3$ & $1$ & $4$ & $1$ & & \\
	$4$ & $1$ & $11$ & $11$ & $1$ & \\
	$5$ & $1$ & $26$ & $66$ & $26$ & $1$
\end{tabular}
\end{center}
\caption{The Eulerian numbers $A(n,k)$ for $n \leq 5$.}
\label{tab:eulerian}
\end{table}

Table~\ref{tab:eulerian} shows the first several rows of the Eulerian
numbers; the reader has already met the rows for $n = 3$ and $n = 4$,
as the counts in Figures~\ref{fig-n3}
and~\ref{fig-n4}.

Each row of this table is symmetric; the descent interpretation yields a straightforward
proof.  Sending
\[
	\pi(1) \pi(2) \cdots \pi(n) \longmapsto (n + 1 - \pi(1))(n + 1 - \pi(2)) \cdots (n + 1 - \pi(n))
\]
reverses the comparison between each pair of adjacent entries, turning
descents into non-descents and vice versa, and thereby gives a
bijection between permutations with $k$ descents and permutations with
$n - 1 - k$ descents. Therefore, the Eulerian numbers are symmetric:
\begin{equation}
  \label{eq:euler-symmetry}
A(n,k) = A(n, n-1-k).
\end{equation}

Our proof relating excedances and descents, by matching recurrences,
leaves something to be desired, because it does not tell us which
permutation with $k$ descents corresponds to which permutation with
$k$ excedances. A classical bijection of Foata~\cite{Foata1965},
developed in detail under the name ``la transformation fondamentale''
by Foata and Sch\"utzenberger~\cite{FoSc}, does exactly that. The two
recurrence proofs essentially illustrate such a bijection, as each
builds a permutation in $S_n$ from a permutation in $S_{n-1}$ and a
choice of $i \in [n]$, and matching these choices recursively yields a
bijection $\Phi\colon S_n \to S_n$ with $\exc(\pi) = \des(\Phi(\pi))$.


While excedances and descents are the interpretations of the Eulerian
numbers used in this paper, there are many others.  For example, if
$Y_1, \ldots, Y_n$ are independent and uniformly random in $(0, 1)$,
then $\floor{Y_1 + \cdots + Y_n} = k$ with probability $A(n,k)/n!$;
the Eulerian numbers also count the regions of several natural
hyperplane arrangements. Petersen's book~\cite{petersen} surveys these
interpretations and more.

%
%

\section{Analyzing the algorithm's performance.}
\label{sec:eulerian-connect}

We are now ready to connect the Eulerian numbers to the performance of \basicstrat.

\begin{theorem}\label{thm:main}
The number of permutations solved in exactly $k+1$ guesses by \basicstrat is the Eulerian number $A(n,k)$.
\end{theorem}

First, we describe \basicstrat informally so we can see how excedances
enter the discussion.  As \basicstrat runs, values cycle rightward until they come to the
correct position, and then they stay put.  If a value starts out to the left of its
correct position, it will move straight there.  If it starts to the right of where it belongs,
it must move all the way to the right edge, wrap around to the left, and keep going from
there.  Comparing any two consecutive guesses, exactly one incorrect value wraps around,
and this continues until we determine the secret.

Which are the values that must wrap around?  A value $\sigma(i)$ wraps
around if its starting location $\sigma(i)$ is to the right of its
eventual destination $i$; that is, if $\sigma(i) > i$.  These are
precisely the values sitting at the excedances of $\sigma$.

\begin{lemma}\label{lemma:one-perm}
  The number of guesses \basicstrat takes to find the secret permutation
  $\sigma$ is exactly $1 + \exc(\sigma)$.
\end{lemma}

Lemma~\ref{lemma:one-perm} immediately implies Theorem~\ref{thm:main}.

\begin{proof}[Proof of Lemma~\ref{lemma:one-perm}]
Let $\sigma$ be the setter's permutation.  Run \basicstrat, producing a series of
guesses $\gamma_1$, $\gamma_2$, $\ldots \in S_n$.  (Part of our task is to show that
\basicstrat terminates with some $\gamma_r = \sigma$.)
In each round, let $R_k$ be the set of values placed farther to the right in $\gamma_k$
than in $\sigma$:
$$R_k := \{ \gamma_k(i) \mid \gamma_k(i) > \sigma(i) \}.$$
Note that $R_1 = \{\sigma(i) \mid \sigma(i) > i\}$, so $R_1 = \Exc(\sigma)$.

 How does $R_{k+1}$ differ from $R_k$?  Recaull that $U_k$ is the set of positions where
 $\sigma$ and $\gamma_k$ disagree; assume that $\gamma_k \ne \sigma$,
 and hence $|U_k| \ge 2$
 (it is impossible for $\gamma_k$ and $\sigma$ to agree in exactly $n-1$ positions).
 Enumerate the positions in $U_k$ as $a_1 < \cdots < a_\ell$.

 Let $v$ denote the value $\gamma_k(a_\ell)$.  The correct position for $v$ must be
 some $a_i$ with $i < \ell$, so $v \in R_k$.  However, by the definition of \basicstrat,
 $v$ wraps around: $v = \gamma_{k+1}(a_1)$, and hence $v \notin R_{k+1}$.

 Every other value in $R_k$ moves farther to the right, so it remains in $R_{k+1}$.  If a
 value is not in $R_k$, then either it is correct (in which case $\gamma_{k+1}$ leaves it
 where it is); or it moves into the correct position; or it moves rightward but not all the way
 to the correct position.  In each case, that value is not in $R_{k+1}$.  We conclude that
 $R_{k+1} = R_k \setminus \{v\}$.
 This evolution is illustrated in Figure~\ref{fig-n9-circles}, a version of
Figure~\ref{fig-n9} in which the values in each $R_k$ are circled.
We conclude that, as long $\gamma_k \ne \sigma$, $|R_{k+1}| = |R_k| - 1$.

 \begin{figure}
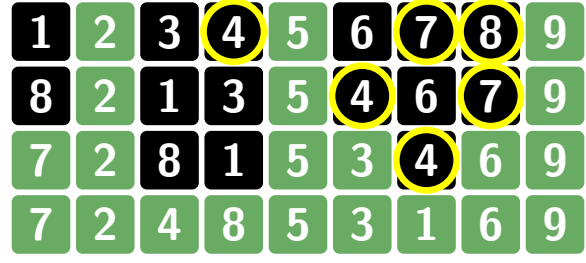

   \begin{center}
  \begin{wordle}{724853169}
          [{
              \draw[line width=2.5pt, color=highlightcircle] (W-1-4) circle (.4cm);
              \draw[line width=2.5pt, color=highlightcircle] (W-1-7) circle (.4cm);
              \draw[line width=2.5pt, color=highlightcircle] (W-1-8) circle (.4cm);
              \draw[line width=2.5pt, color=highlightcircle] (W-2-6) circle (.4cm);
              \draw[line width=2.5pt, color=highlightcircle] (W-2-8) circle (.4cm);
              \draw[line width=2.5pt, color=highlightcircle] (W-3-7) circle (.4cm);
          }]
    123456789
    821354679
    728153469
    724853169
  \end{wordle}
   \end{center}
   \caption{The deterministic strategy \basicstrat guessing the permutation
     $\sigma = 724853169$.  We circle the values in each set $R_k$:
     the excedances of $\sigma$ are the positions $1$, $3$, and $4$, so the values
     $7$, $4$, and $8$ each wrap around, and the strategy succeeds in $1 + 3 = 4$ rounds.}
     \label{fig-n9-circles}
 \end{figure}

 Clearly, if $\gamma_k = \sigma$, then $R_k = \emptyset$.  The converse also holds:
 if $R_k = \emptyset$, then, for all $i$,
 we have $\gamma_k(i) \le \sigma(i)$.  Summing over all
 $i$, we have $\sum_i \gamma_k(i) \le \sum_i \sigma(i)$; both $\gamma_k$ and $\sigma$
 are permutations, so both sums equal $1 + \cdots + n$, and hence
 each inequality must be an equality and $\gamma_k = \sigma$.

 Let $r = 1 + \exc(\sigma)$.  For $k < r$, $|R_k| = r - k > 0$, so $\gamma_k \ne \sigma$,
 and \basicstrat continues.  After $r$ rounds, $|R_r| = 0$ and $\gamma_r = \sigma$.
 We conclude that \basicstrat solves $\sigma$ in exactly $1 + \exc(\sigma)$ guesses.
\end{proof}

Theorem~\ref{thm:main}, combined with the symmetry~\eqref{eq:euler-symmetry},
gives a simple description of the average-case performance of \basicstrat:

\begin{corollary}
\label{cor:expected:guesses}
If the secret $\sigma \in S_n$ is chosen uniformly at random, then the expected number of guesses for \basicstrat to solve $\sigma$ is $(n+1)/2$.
\end{corollary}

%
%

\section{Proof of optimality.}
\label{sec:proof}

We now prove that \basicstrat is optimal in a strong sense: for every $r$, \basicstrat solves at least as many secret permutations within~$r$ guesses as any other strategy. The main technical statement is a deterministic one:

\begin{theorem}
\label{thm:opt}
For every deterministic strategy and every $d \geq 0$, the number of permutations in $S_n$ that the strategy solves within $d+1$ guesses is at most $\sum_{k=0}^{d} A(n,k)$.
\end{theorem}

The sum $\sum_{k=0}^{d} A(n,k)$ counts the permutations with at most
$d$ excedances, so for $d \geq n-1$ it counts all of $S_n$ and the
bound is trivial. For every $d$, the bound is sharp, because
Theorem~\ref{thm:main} says that \basicstrat{} solves precisely the
permutations with at most $d$ excedances within $d+1$ guesses. Thus, no
deterministic strategy, even one that guesses non-permutations, beats
\basicstrat. The corresponding statement for randomized strategies
follows by a standard derandomization argument.


\begin{theorem}
\label{thm:opt-rand}
Given a strategy $\mathcal{S}$, possibly randomized, and a uniformly random secret $\sigma \in S_n$, let $P_r(\mathcal{S})$ denote the probability that $\mathcal{S}$ solves $\sigma$ within $r$ guesses. Then $P_r(\basicstrat) \geq P_r(\mathcal{S})$ for every $r \geq 1$ and every $\mathcal{S}$.
\end{theorem}

\begin{proof}[Proof (assuming Theorem~\ref{thm:opt})]
Fix $r \geq 1$. Before the game begins, the guesser can flip all the coins in advance, recording an outcome for every decision the strategy might ever face. Fixing these outcomes turns $\mathcal{S}$ into a deterministic strategy, and Theorem~\ref{thm:opt} applies to it. However the coins land, the resulting deterministic strategy solves at most as many secrets within $r$ guesses as \basicstrat{} does, and hence, against a uniformly random secret, succeeds with probability at most $P_r(\basicstrat)$. There is no way for the coins to land that beats \basicstrat, so averaging over the coin flips, $P_r(\mathcal{S}) \leq P_r(\basicstrat)$.
\end{proof}

Theorem~\ref{thm:opt-rand} also determines the optimal expected number of guesses.

\begin{corollary}
\label{cor:lower-bound}
Against a uniformly random secret $\sigma \in S_n$, every strategy requires at least $(n+1)/2$ guesses in expectation, and \basicstrat{} achieves this.
\end{corollary}

\begin{proof}
Fix a strategy, let $T$ denote the number of guesses it takes, and let $T^*$ denote the number \basicstrat{} takes; Corollary~\ref{cor:expected:guesses} gives $\mathbb{E}[T^*] = (n+1)/2$. By Theorem~\ref{thm:opt-rand}, $\Pr(T > r) \geq \Pr(T^* > r)$ for every $r$, and summing over $r \geq 0$ yields
\[
	\mathbb{E}[T] = \sum_{r \geq 0} \Pr(T > r) \geq \sum_{r \geq 0} \Pr(T^* > r) = \mathbb{E}[T^*] = \frac{n+1}{2}.
\]
(If $T$ is not almost surely finite, then $\mathbb{E}[T] = \infty$ and the inequality holds trivially.)
\end{proof}

These guarantees concern a uniformly random secret, but the guesser can enforce them against any setter whatsoever. Recall from Section~\ref{sec:random-secret} that a guesser who relabels every guess by a uniformly random permutation $\tau$ faces, in effect, the uniformly random secret $\tau^{-1}\sigma$, no matter how $\sigma$ was chosen. Randomized in this way, \basicstrat{} solves every fixed secret, even one chosen adversarially, within $r$ guesses with probability $\sum_{k=0}^{r-1} A(n,k)/n!$, and in $(n+1)/2$ expected guesses. Nor can the guesser do any better, since the setter is free to choose the uniform distribution. In the language of game theory, these are the minimax values of \classic, and randomized \basicstrat{} is an optimal strategy for the guesser.


\subsection*{Feedback histories.}

Theorem~\ref{thm:opt} concerns deterministic strategies, so we fix one for the remainder of the section. We also adopt a bookkeeping convention: once our strategy has solved the secret, we let it continue guessing that same permutation forever. The game then continues indefinitely, so every secret produces feedback of every length, and the convention does not change the round in which any secret is first solved.

A \emph{feedback history} of length $d$ is a sequence $U=(U_1,\dots,U_d)$ of subsets of $[n]$, which we think of as the feedback from $d$ rounds of play, with $U_j$ recording which positions were guessed correctly in round~$j$. Such a history is \emph{feasible} if our strategy produces it against some secret permutation. We write $\gamma_j^U$ for the guess our strategy makes in round $j$ after receiving the feedback $U_1, \dots, U_{j-1}$. Note that because the strategy is deterministic, $\gamma_j^U$ depends only on these first $j - 1$ sets of $U$, and not on the later ones. For each feasible history $U$, we choose once and for all a secret permutation that produces $U$, and we call it $\sigma_U$, so that
\[
	U_j=\{i:\gamma_j^U(i)=\sigma_U(i)\}
	\qquad
	\text{for every } 1\le j\le d.
\]

Each feasible history of length $d$ accounts for at most one secret solved within $d+1$ guesses. To see this, note that the guess $\gamma_{d+1}^U$ is determined by the feasible history $U$. If a secret producing $U$ is solved in round $d+1$, then it equals $\gamma_{d+1}^U$ because that guess solved it, and if it was solved in an earlier round, then it equals $\gamma_{d+1}^U$ because the strategy repeats a solved secret forever. To prove Theorem~\ref{thm:opt}, it suffices to prove the bound
\[
	|\{\text{feasible histories of length } d\}| \leq \sum_{k=0}^{d} A(n,k).
\]
From here on, the strategy is implicit. We work only to bound the number of feasible histories.

To do so, we define, for each feasible history, a polynomial in the
variables $x_1, \dots, x_n$ of degree at most $d$ in each variable. We
regard these polynomials as functions on the $n!$ points of
$\mathbb{R}^n$ whose coordinates are $1, 2, \ldots, n$ in some order,
and we show that the functions corresponding to distinct feasible histories are
linearly independent. Here, we use the hypothesis that the secret is a
permutation. Each secret can be encoded as a point of this small set,
while the guesses, which may be arbitrary words, determine the
coefficients of the polynomials. The number of feasible histories is
therefore at most the dimension of the space $V_d$ of functions on
these points arising from polynomials obeying our degree bound, and we
compute this dimension to be precisely $\sum_{k=0}^{d} A(n,k)$, using
a basis of descent monomials due to Garsia and Stanton.

%
%

\subsection*{From histories to polynomials.}

Bounding the cardinality of a set by assigning linearly independent polynomials to its elements and computing the dimension of a space containing them is a standard method in combinatorics. It appears frequently in extremal problems; see, for example, Babai and Frankl's legendary notes~\cite{babai:linear-algebra-:}, or the books of Guth~\cite{guth:polynomial-meth:} or Matou\v{s}ek~\cite{matouv-sek:thirty-three-mi:}.

We begin by encoding permutations as points. Pick $n$ distinct real numbers $v_1, v_2, \ldots, v_n$, and encode each $\pi \in S_n$ as the point
\[
	p_\pi = (v_{\pi(1)}, v_{\pi(2)}, \ldots, v_{\pi(n)}) \in \mathbb{R}^n.
\]
Let $X_n$ be the set of these $n!$ points, one for each way to arrange $v_1, \ldots, v_n$ as a tuple.

\begin{proposition}
\label{prop:full-space}
The space of functions on~$X_n$ has dimension $n!$, and it has a basis of restrictions of polynomials.
\end{proposition}

\begin{proof}
For each $\pi \in S_n$, consider the Lagrange interpolation polynomial
\[
	L_\pi(x) = \prod_{i=1}^{n} \prod_{\substack{m \in [n] \\ m \neq \pi(i)}} \frac{x_i - v_m}{v_{\pi(i)} - v_m}.
\]
The value $L_\pi(p_\tau)$ is $1$ when $\tau = \pi$ and $0$ otherwise, so $L_\pi$ restricts to the indicator function of the point $p_\pi$. The indicator functions of the $n!$ points of $X_n$ form a basis for the space of functions on $X_n$, and this basis consists of restrictions of the polynomials~$L_\pi$.
\end{proof}
 
In the following, the choice of the values $v_i$ does not matter, so long as they are distinct, so we simplify notation by choosing $v_1 = 1$, $v_2 = 2$, $\ldots$, $v_n = n$. For a polynomial~$P$ and a permutation $\pi$, we also write~$P(\pi)$ as shorthand for~$P(p_\pi)$.

%
%

\subsection*{The history polynomials.}

For each feasible history $U = (U_1, \dots, U_d)$, define
\[
	P_{U,j} = \prod_{i \notin U_j} \bigl(x_i - \gamma_j^U(i)\bigr),
	\qquad
	P_U = \prod_{j=1}^{d} P_{U,j}.
\]
The polynomial $P_{U,j}$ records the wrong guesses in round $j$, one factor for each incorrect position, and the product $P_U$ records the wrong guesses across the whole history. To spell this out, a position $i \notin U_j$ is precisely one where the guessed value $\gamma_j^U(i)$ is not the secret's value at~$i$, and the factor $(x_i - \gamma_j^U(i))$ vanishes exactly on the points of $X_n$ that agree with this wrong guess. If the strategy solves its secret before round $j$, then $U_j = [n]$ and $P_{U,j}$ is an empty product, equal to the constant $1$.

Before going on, we pause for an example with $n = 3$ and $d = 1$, a history consisting of a single round. If the strategy's first guess is $\gamma_1 = 123$, then the feasible histories of length one are
\[
\begin{array}{c|c|c}
	U_1 & \text{secrets producing it} & P_U \\ \hline
	\{1,2,3\} & 123 & 1 \\
	\{1\} & 132 & (x_2 - 2)(x_3 - 3) \\
	\{2\} & 321 & (x_1 - 1)(x_3 - 3) \\
	\{3\} & 213 & (x_1 - 1)(x_2 - 2) \\
	\varnothing & 231 \text{ or } 312 & (x_1 - 1)(x_2 - 2)(x_3 - 3)
\end{array}
\]
Two agreements with a permutation force a third, so the feedbacks $\{1,2\}$, $\{1,3\}$, and $\{2,3\}$ are infeasible, and there are five feasible histories in all. If the strategy instead opens with the word $\gamma_1 = 111$, then the secret agrees with it precisely where the secret's value is $1$, and every permutation has exactly one such position, so only three histories are feasible:
\[
\begin{array}{c|c|c}
	U_1 & \text{secrets producing it} & P_U \\ \hline
	\{1\} & 123 \text{ or } 132 & (x_2 - 1)(x_3 - 1) \\
	\{2\} & 213 \text{ or } 312 & (x_1 - 1)(x_3 - 1) \\
	\{3\} & 231 \text{ or } 321 & (x_1 - 1)(x_2 - 1)
\end{array}
\]
Either way, the feasible histories partition the secrets, and having more histories is desirable, because it means the strategy separates the secrets better. Here the guess $123$ splits the six secrets into five classes, four of which are already singletons, while $111$ splits them into three classes of two, so the first guess $123$ is more informative. In each case, there are at most five feasible histories, matching the dimension $\dim V_1 = 5$ that we compute below.

We need two properties of $P_U$. First, within any round~$j$, each variable~$x_i$ appears in at most one factor of $P_{U,j}$, so across all $d$ rounds, $\deg_{x_i} P_U \leq d$ for each $i \in [n]$. Second, $P_U(\sigma_U) \neq 0$, because each $P_{U,j}$ multiplies only over the positions $i \notin U_j$, and these are precisely the positions where $\sigma_U(i) \neq \gamma_j^U(i)$, so every factor is nonzero at $\sigma_U$.

\begin{proposition}
\label{prop:indep}
The polynomials $\{P_U \mid U \text{ a feasible history of length } d\}$ are linearly independent
as functions on $X_n$.
\end{proposition}

\begin{proof}
  We evaluate each polynomial $P_U$ at each of the chosen secrets $\sigma_{U'}$ and show that the resulting matrix, with rows and columns suitably ordered, is triangular with nonzero diagonal entries.
  As a consequence, the polynomials are linearly independent
  as functions on $X_n$.

  For the ordering, sort the feasible histories by their sequences of
  feedback sizes $(|U_1|, |U_2|, \dots, |U_d|)$ lexicographically,
  breaking ties arbitrarily. The diagonal entries $P_U(\sigma_U)$
  are nonzero, as observed above. It remains to show that if $U \neq U'$
  and $P_U(\sigma_{U'}) \neq 0$, then $U'$ comes strictly before $U$ in
  our order.

  Assume the premise.
  Let $s$ be the first round in which $U$ and $U'$ differ,
  so that $U_j = U_j'$ for $j < s$ and $U_s \neq U_s'$.
  The strategy's guesses depend only on the preceding feedback,
  so $\gamma_j^U = \gamma_j^{U'}$ for every $j \leq s$.
  Since, by our assumption, $P_U(\sigma_{U'}) \neq 0$,
  every factor of $P_U$ is nonzero at $\sigma_{U'}$.
  In particular, for each $i \notin U_s$,
  the round-$s$ factor $P_{U,s}$ records $\sigma_{U'}(i) \neq \gamma_s^U(i)$.
  Because the round-$s$ guesses for the two histories,
  $\gamma_s^{U}$ and $\gamma_s^{U'}$, are equal,
  the secret $\sigma_{U'}$ disagrees with $\gamma_s^{U'}$
  at every position outside $U_s$.
  Therefore, every position outside $U_s$ is also outside $U_s'$.
  In other words, $U_s' \subseteq U_s$
  and the containment is strict because $U_s \neq U_s'$.
  Thus, $|U_s'| < |U_s|$.  Since $U_j = U_j'$ for $j < s$,
  the history $U'$ comes strictly before $U$ in our order.
\end{proof}

%
%

\subsection*{The function space~$V_d$.}

We now define the function space where we work, which consists of the functions on $X_n$ that polynomials of low degree can produce. Let
\[
	V_d = \bigl\{ P|_{X_n} \mid P \in \mathbb{R}[x_1, \dots, x_n] \text{ and } \deg_{x_i} P \leq d \text{ for every index } i \bigr\}.
\]
Because $\deg_{x_i} P_U \leq d$ for every $i$, the restrictions $P_U|_{X_n}$ all lie in~$V_d$.  By Proposition~\ref{prop:indep},
these restrictions are linearly independent, so
\[
	|\{\text{feasible histories of length } d\}| \leq \dim V_d.
\]

We can compute the dimension of $V_{n-1}$ immediately.
The Lagrange polynomials $L_\pi$ of Proposition~\ref{prop:full-space}
have degree at most $n-1$ in every variable, so their restrictions lie
in $V_{n-1}$. Since these restrictions span the space of functions on
$X_n$ by Proposition~\ref{prop:full-space}, the space $V_{n-1}$ is the
full space of functions on $X_n$, and $\dim V_{n-1} = n!$.

Before computing $\dim V_d$ for general $d$, we pause to describe a closely related object, the coinvariant algebra of $S_n$. This is the context in which the basis we borrow was first discovered.
While the parallel between that setting and ours is worth seeing,
we emphasize that nothing to follow depends on it; the computation
of $\dim V_d$ is developed from first principles.
Let $e_i = e_i(x_1, \dots, x_n)$ be the elementary symmetric polynomial
of degree~$i$,
\[
	e_i = \sum_{1 \leq j_1 < \dots < j_i \leq n} x_{j_1} \cdots x_{j_i};
\]
for instance, $e_1 = x_1 + \cdots + x_n$ and $e_n = x_1 \cdots x_n$. The fundamental theorem of symmetric polynomials says that every symmetric polynomial can be expressed uniquely as a polynomial in $e_1, \ldots, e_n$. Let $I$ be the ideal $(e_1, \ldots, e_n)$. The \emph{coinvariant algebra} of $S_n$ is the quotient
\[
	C_n = \mathbb{R}[x_1, \dots, x_n] / I.
\]
By the fundamental theorem, every symmetric polynomial with zero constant term lies in~$I$, and so is equivalent to zero in the quotient $C_n$.

Our setting is slightly different. Instead of working modulo the ideal $I$, we study when two polynomials agree as functions on the point set $X_n$, writing $P \equiv_{X_n} Q$ to mean that $P(x) = Q(x)$ for every $x \in X_n$. Symmetric polynomials play the same role here. The coordinates of every point of~$X_n$ are $1, 2, \ldots, n$ in some order, and a symmetric polynomial does not distinguish between orders, so every symmetric polynomial~$S$ is constant on~$X_n$, with $S \equiv_{X_n} S(1, 2, \ldots, n)$. The same situation arises in the study of ``generalized coinvariant algebras'' by Haglund, Rhoades, and Shimozono~\cite[Section~4.1]{haglund:ordered-set-par:}, who study a family of point sets $Y_{n,k}$; their $Y_{n,n}$ with parameters $\alpha_i = i$ is our~$X_n$, though we redevelop the little we need.

Garsia and Stanton~\cite[Section~9]{garsia:group-actions-o:} gave a basis of the coinvariant algebra indexed by permutations. We show that the same construction gives a basis of the space of functions on $X_n$, adapted to the spaces $V_d$ in the following sense. The spaces $V_0 \subseteq V_1 \subseteq \cdots$ sit inside one another (a ``filtration'' of the function space), and the basis elements indexed by permutations with at most $d$ descents form a basis of $V_d$. The \emph{descent monomial} of a permutation $\pi \in S_n$ is
\[
	B_\pi(x) = \prod_{j \in \Des(\pi)} \bigl(x_{\pi(1)} x_{\pi(2)} \cdots x_{\pi(j)}\bigr),
\]
where $\Des(\pi)$ is the descent set defined in Definition~\ref{def:desc}
of Section~\ref{sec:eulerian}.
Each variable $x_{\pi(i)}$ appears in $B_\pi$ once for every descent~$j$ with $j \geq i$, so writing
\[
	d_i(\pi) = |\{j \in \Des(\pi) \mid j \geq i\}|,
\]
the number of descents at or to the right of position $i$, we have
$$B_\pi = x_{\pi(1)}^{d_1(\pi)} x_{\pi(2)}^{d_2(\pi)} \cdots x_{\pi(n)}^{d_n(\pi)}.$$
The exponents satisfy $d_1(\pi) \geq d_2(\pi) \geq \cdots \geq d_n(\pi) = 0$, and the largest, $d_1(\pi)$, counts all descents of $\pi$. Therefore, when $\des(\pi) \leq d$, we have $B_\pi|_{X_n} \in V_d$.

For example, the six descent monomials for $n = 3$, listed by number of descents, are
\[
\begin{array}{c|c|c}
	\pi & \Des(\pi) & B_\pi \\ \hline
	123 & \varnothing & 1 \\ \hline
	213 & \{1\} & x_2 \\
	312 & \{1\} & x_3 \\
	132 & \{2\} & x_1 x_3 \\
	231 & \{2\} & x_2 x_3 \\ \hline
	321 & \{1, 2\} & x_3 \cdot x_3 x_2 = x_3^2 x_2
\end{array}
\]
We observe that $V_0$ has the single basis element $1$, the space $V_1$ has the five basis elements with $\des(\pi) \leq 1$, and $V_2$ has all six. The restriction $x_1|_{X_3}$ lies in $V_1$, but $x_1$ is not itself a descent monomial, so the claimed basis must express it another way. Since $e_1 = 6$ on $X_3$, we have $x_1 = (e_1 - 6) + 6 B_{123} - B_{213} - B_{312}$, and therefore $x_1 \equiv_{X_3} 6 B_{123} - B_{213} - B_{312}$. The proof of Proposition~\ref{prop:basis} iterates this kind of substitution, trading each monomial that is not a descent monomial for smaller ones, in a sense we now make precise.

The argument requires a total order on monomials. We compare two monomials by first comparing their exponent vectors sorted into decreasing order, lexicographically; if the sorted vectors agree, we compare the exponent vectors themselves lexicographically. The effect of the tie-breaker is that between two arrangements of the same exponents, the larger monomial places its larger exponents on smaller-indexed variables. For example, with $n = 3$,
\[
	x_3^2 x_2 \prec x_1^2 x_2
\]
because both have sorted exponent vector $(2, 1, 0)$, and the tie is broken by the exponent vectors themselves: $(0, 1, 2) <_{\mathrm{lex}} (2, 1, 0)$. Also,
\[
	x_1 x_2 x_3 \prec x_1^2 x_2
\]
because the sorted exponent vectors are $(1, 1, 1)$ and $(2, 1, 0)$, and $(1, 1, 1) <_{\mathrm{lex}} (2, 1, 0)$. This is not a \emph{monomial order} in the sense of Gr\"obner basis theory, as it is not compatible with multiplication, but the argument does not require that. The minimal monomial is the constant $1$, and all monomials arising in the argument have every exponent at most $d$, so the induction runs over a finite set.

Under this order, the substitutions we iterate trade the maximal monomial of a suitable symmetric polynomial for smaller ones. For instance, when $n = 3$, the monomial $x_1 x_2$ is the maximal monomial of
\[
	e_2 = x_1 x_2 + x_1 x_3 + x_2 x_3.
\]
Since $e_2 = 1\cdot 2 + 1\cdot 3 + 2\cdot 3 = 11$ on $X_3$, we have
\[
	x_1 x_2 \equiv_{X_3} 11 - x_1 x_3 - x_2 x_3 = 11 B_{123} - B_{132} - B_{231},
\]
trading the non-basis monomial $x_1 x_2$ for smaller descent monomials.

This straightening argument is not new. Allen~\cite[Section~2]{allen:the-descent-mon:} uses it to give an elementary proof that the descent monomials form a basis of the coinvariant algebra $C_n$; see also Adin, Brenti, and Roichman~\cite[Sections~2.5--3.4]{adin:descent-represe:}. Where the coinvariant setting discards symmetric polynomials with zero constant term, as these are zero in $C_n$, we discard the difference between a symmetric polynomial and its value on $X_n$, as these agree on $X_n$.

\begin{proposition}
\label{prop:basis}
The functions $\{B_\pi|_{X_n} \mid \text{$\pi \in S_n$ and $\des(\pi) \leq d$}\}$ form a basis of~$V_d$.
Consequently, $\dim V_d = \sum_{k=0}^{d} A(n,k)$.
\end{proposition}

\begin{proof}
Proposition~\ref{prop:full-space} shows that $\dim V_{n-1} = n!$, and there are exactly $n!$ descent monomials $B_\pi$, one for each $\pi \in S_n$. Therefore, if the restrictions of the descent monomials span $V_{n-1}$, then they form a basis of it, and every subset of them is linearly independent. Define
\[
	W_d = \langle B_\pi|_{X_n} \mid \text{$\pi \in S_n$ and $\des(\pi) \leq d$} \rangle.
\]
Since every element of $V_d$ is the restriction of a linear combination of monomials with all exponents at most $d$, it suffices to prove that every such monomial restricts to an element of $W_d$.

We give a proof by induction along our order on monomials. The minimal monomial is the constant $1$, which is the descent monomial of the identity permutation and so lies in every $W_d$. Let
\[
	M = x_1^{a_1} x_2^{a_2} \cdots x_n^{a_n}
\]
be a monomial with all $a_i \leq d$, and assume that for all monomials $M' \prec M$ with exponents at most $d$, we have $M'|_{X_n} \in W_d$. We show there is a descent monomial $B_\pi$ with $B_\pi|_{X_n}\in W_d$ and a symmetric polynomial $S$ such that $M$ is the maximal monomial in the product $B_\pi S$.

Define a permutation $\pi$ so that
\[
	a_{\pi(1)} \ge a_{\pi(2)} \ge \cdots \ge a_{\pi(n)},
\]
breaking ties so that where $a_{\pi(i)} = a_{\pi(i+1)}$, we have $\pi(i) < \pi(i+1)$. This convention means that a descent $\pi(i) > \pi(i+1)$ implies $a_{\pi(i)} > a_{\pi(i+1)}$.

Consider the descent monomial
\[
	B_\pi
	=
	x_{\pi(1)}^{d_1} x_{\pi(2)}^{d_2} \cdots x_{\pi(n)}^{d_n},
\]
where we abbreviate $d_i = d_i(\pi)$. The differences
\[
	c_j = a_{\pi(j)} - d_j
\]
form a weakly decreasing nonnegative sequence $c_1, c_2, \ldots, c_n$. To see this, first note that since $d_n = 0$, we have $c_n = a_{\pi(n)} \geq 0$. If~$j$ is a descent of $\pi$, then $d_j = d_{j+1}+1$ and $a_{\pi(j)}\geq a_{\pi(j+1)}+1$, and otherwise, $d_j = d_{j+1}$ and $a_{\pi(j)} \geq a_{\pi(j+1)}$. Either way, $c_j \geq c_{j+1}$.

As a consequence, $B_\pi$ divides~$M$, and we can write $M = B_\pi N$ for the monomial
\[
	N = x_{\pi(1)}^{c_1} x_{\pi(2)}^{c_2} \cdots x_{\pi(n)}^{c_n}.
\]
Since $c_1 \geq 0$, we have
\[
	\des(\pi) = d_1 \le a_{\pi(1)} \le d,
\]
showing that $B_\pi|_{X_n} \in W_d$.

Form the symmetric polynomial $S$ whose terms are the distinct monomials of the form
\[
	x_{\tau(1)}^{c_1} x_{\tau(2)}^{c_2} \cdots x_{\tau(n)}^{c_n}
\]
with $\tau \in S_n$. (This is the \emph{monomial symmetric polynomial} indexed by the partition $(c_1,c_2,\dots,c_n)$.)

We show the monomial $B_\pi N$ is maximal among the monomials of $B_\pi S$. Suppose that $B_\pi N'$ is one of these monomials and that it is maximal. For some permutation $\rho$,
\[
	B_\pi N' = \prod x_{\pi(i)}^{d_i + c_{\rho(i)}}.
\]
First we argue that we may take $\rho$ to be the identity. Suppose instead that some pair of indices $i < j$ has $c_{\rho(i)} < c_{\rho(j)}$. Swapping these two values of $c$ replaces the exponents $d_i + c_{\rho(i)}$ and $d_j + c_{\rho(j)}$ with $d_i + c_{\rho(j)}$ and $d_j + c_{\rho(i)}$. Since $d_i \geq d_j$, this swap pairs the two larger summands together, so it weakly increases the larger of the two exponents while preserving their sum, and hence weakly increases the sorted exponent vector. Performing such swaps until none remains, the maximality of $B_\pi N'$ is never disturbed, and we arrive at the exponents $d_i + c_i$ of $B_\pi N$. The sorted exponent vector of $B_\pi N'$ therefore equals that of $B_\pi N$, and the swaps that preserve it exactly are those within positions sharing a common value of $d_i$. So $N'$ can differ from $N$ only in how it distributes the $c_i$ among such positions. These positions form a run containing no descent of $\pi$, so $\pi$ increases along the run, and $N$ assigns its weakly decreasing exponents to variables of increasing index. Among these arrangements, $B_\pi N$ is therefore the one placing its larger exponents on smaller-indexed variables, which the tie-break makes largest, so $B_\pi N' = B_\pi N$.

Since $S$ is symmetric, its restriction to $X_n$ is the constant $S(1,2,\ldots,n)$. Let $S' = S - N$. Restricting to~$X_n$, we have
\[
	M
	=
	B_\pi N
	\equiv_{X_n}
	B_\pi S(1,2,\dots,n) - B_\pi S',
\]
and the term $B_\pi S(1, 2, \dots, n)$ is a scalar multiple of the descent monomial $B_\pi$, whose restriction lies in $W_d$. As for $B_\pi S'$, we have shown that every monomial $B_\pi N'$ of $B_\pi S'$ satisfies $B_\pi N' \prec M$, and each has largest exponent at most that of $M$, hence at most $d$. By the inductive hypothesis, each $B_\pi N'|_{X_n}$ lies in $W_d$, so $B_\pi S'|_{X_n} \in W_d$, and therefore $M|_{X_n} \in W_d$. This completes the induction, and with it the proof.
\end{proof}

The pieces of Theorem~\ref{thm:opt} are now all in place.
Each deterministic strategy solves at most one secret per
feasible feedback history of length $d$.
Proposition~\ref{prop:indep} shows the history polynomials are
linearly independent in $V_d$.
Proposition~\ref{prop:basis} shows
$\dim V_d = \sum_{k=0}^{d} A(n,k)$.
So, a deterministic strategy solves at most $\sum_{k=0}^{d} A(n,k)$
secrets within $d+1$ guesses.
With Theorem~\ref{thm:opt} proved,
Theorem~\ref{thm:opt-rand} and its consequences hold,
proving that \basicstrat{} is optimal for \classic.

%
%

\section{Connections with Mastermind.}
\label{sec:mastermind}

There are natural connections between Wordle and
Mordecai Meirowitz's classic game Mastermind.
In Mastermind, the setter chooses a ``code'': a string of length~$m$
from the alphabet $[n]$.
In each round, the guesser tries a string in $[n]^m$,
and the setter responds with two numbers:
the number of entries that are correct (right number, right place)
and the number of additional entries that are
misplaced (right number, wrong place).
This is a similar setup to Wordle, except that in
Wordle we find out specifically which entries are correct or
misplaced, instead of just a count.

Donald Knuth~\cite{Knuth1976} showed that, for the original game
($m = 4$ and $n = 6$),
five rounds are sufficient for the guesser to guess
the code. Many others have evaluated the game for various values
of~$m$ and~$n$. The typical emphasis is on a worst-case analysis; one
can imagine the setter adversarially altering the code to be
consistent with all information presented thus far, with the goal of
maximizing the number of rounds. Worst-case analysis is natural here
because the symmetrization of Section~\ref{sec:random-secret} does
not apply. In \classic, relabeling values turns any secret into any
other, and this is why the guesser can force even an adversarial
setter to behave, in effect, like the uniform distribution. In
Mastermind, relabeling colors and permuting positions can never turn
the code $(1,2,3,4)$, which has no repeated entries, into the code
$(1,1,2,3)$, which has a repeat. Guarantees against a random
code therefore say nothing about guarantees against a particular
setter, and the literature focuses on the worst case. Seen from this
angle, requiring the secret in \classic{} to be a permutation is not
merely a simplification; it puts all secrets into a single equivalence
class, and this is what lets the guesser reduce any setter's strategy to a
uniformly random one.

One Mastermind variant is ``Permutation Mastermind,'' in which $m = n$
and both the setter's code and the guesser's guesses must contain each
value exactly once. (In this variant, the setter responds with the
number $c$ of correct entries, since the number misplaced is simply
$n-c$.) El Ouali, Glazik, Sauerland, and Srivastav~\cite{EOGlSaSr2018}
showed that $n \log_2 n + O(n)$ queries suffice. Martinsson and
Su~\cite{martinsson-su:mastermind} later showed that a guesser who may
guess arbitrary words needs only $O(n)$ queries; as each response
conveys at most $\log_2(n+1)$ bits and the secret is one of $n!$
permutations, this is optimal up to the constant factor. Afshani,
Agrawal, Doerr, Doerr, Larsen, and Mehlhorn~\cite{AADDLM2019} analyze
the query complexity of a problem rooted in a Mastermind-like game
with a secret permutation.

Li and Zhu~\cite{LiZhu2024} analyze ``Clear Mastermind,'' which is
Mastermind with Wordle's feedback, including its conventions for
repeated letters. When the secret is a permutation, they observe that
this feedback conveys nothing beyond the set of correct positions,
since every letter is known to occur exactly once; their game is
then precisely \classic. They show that the guesser needs $n$ guesses
in the worst case, against a setter who may adaptively change the
secret. This bound also follows from our Theorem~\ref{thm:opt}: since
$A(n, n-1) = 1$, whatever the strategy, some secret permutation
requires $n$ guesses even when the setter commits to it in advance.

Li and Zhu also give a simple strategy achieving this worst case: the
first~${n-1}$ guesses are the constant strings $1^n, 2^n, \ldots,
(n-1)^n$, at which point the guesser knows the secret and solves in
one more guess. This strategy achieves its worst case with admirable
consistency: it solves every one of the $n!$ secrets in exactly $n$
guesses, and so averages $n$ guesses as well, while \basicstrat matches its
worst case and averages $(n+1)/2$.

%
%

\section{Further work.}
\label{sec:open}

The game \classic{} admits natural variations.

First, one can generalize from the ``monochrome'' case to a ``multicolor'' version of the problem. Imagine that, instead of just choosing a permutation of $[n]$, the setter chooses an arrangement of~$n$ cards from a larger deck: each card has a rank between~$1$ and~$n$ and one of~$s$ colors. We still require that the setter's arrangement contain each rank precisely once. There is a natural extension of \basicstrat{} to this setting, and we conjecture that
the strategy is again optimal. Also, the numbers that arise from \rainbow correspond to higher-order Eulerian numbers~\cite{steingrimsson}. See~\cite{arxiv-kutin-smithline} for more details.

Second, one can vary the feedback. Mastermind's numerical feedback and
Wordle's positional feedback are two points on a spectrum, and the
number of guesses needed to solve a secret permutation is a measure of
how much each feedback reveals. Our results give the exact answer for
Wordle feedback, both on average and in the worst case. For
Mastermind's feedback, the number of guesses needed grows linearly
in~$n$ by the work of Martinsson and
Su~\cite{martinsson-su:mastermind} discussed in
Section~\ref{sec:mastermind}, but the optimal constant is unknown.
What happens for feedback rules between and beyond these remains open.

Finally, our optimality theorem says that no strategy beats \basicstrat{} for a uniformly random secret, but the analysis of \classic{} against other natural distributions on secrets, or with restricted guesses, remains open. We would be delighted to see what other mathematics these games reveal.

%
%

\section*{Acknowledgments.}
The authors are indebted to Alex Miller for making the connection with excedances and for
proposing the two-colored
variant. We would also like to thank Joe Buhler, Tim Chow, Keith Frankston,
Zachary Hamaker, Ben Howard, Danny Scheinerman, Jeff VanderKam, and Doron Zeilberger for helpful discussions.
And we appreciate the package \texttt{wordle.sty}, by
Andrew Mathas and C{\'e}dric Pierquet, which we used to make the Wordle diagrams
throughout the paper.

%
%
%

\bibliographystyle{plainurl}
\bibliography{pw3}

\begin{thebibliography}{10}

\bibitem{adin:descent-represe:}
Ron~M. Adin, Francesco Brenti, and Yuval Roichman.
\newblock Descent representations and multivariate statistics.
\newblock {\em Trans. Amer. Math. Soc.}, 357(8):3051--3082, 2005.
\newblock \href {https://doi.org/10.1090/S0002-9947-04-03494-4}
  {\path{doi:10.1090/S0002-9947-04-03494-4}}.

\bibitem{AADDLM2019}
Peyman Afshani, Manindra Agrawal, Benjamin Doerr, Carola Doerr, Kasper~Green
  Larsen, and Kurt Mehlhorn.
\newblock The query complexity of a permutation-based variant of {M}astermind.
\newblock {\em Discrete Appl. Math.}, 260:28--50, 2019.
\newblock \href {https://doi.org/10.1016/j.dam.2019.01.007}
  {\path{doi:10.1016/j.dam.2019.01.007}}.

\bibitem{allen:the-descent-mon:}
Edward~E Allen.
\newblock The descent monomials and a basis for the diagonally symmetric
  polynomials.
\newblock {\em J. Algebraic Combin.}, 3(1):5--16, 1994.
\newblock \href {https://doi.org/10.1023/A:1022481303750}
  {\path{doi:10.1023/A:1022481303750}}.

\bibitem{babai:linear-algebra-:}
L{\'a}szl{\'o} Babai and P{\'e}ter Frankl.
\newblock {\em Linear Algebra Methods in Combinatorics}.
\newblock 2024.
\newblock Unpublished manuscript.
\newblock URL:
  \url{https://people.cs.uchicago.edu/~laci/babai-frankl-book2024.pdf}.

\bibitem{Benveniste2022Wordle}
Alexis Benveniste.
\newblock Wordle is a love story.
\newblock {\em The New York Times}, 2022.
\newblock URL:
  \url{https://www.nytimes.com/2022/01/03/technology/wordle-word-game-creator.html}.

\bibitem{EOGlSaSr2018}
Mourad El~Ouali, Christian Glazik, Volkmar Sauerland, and Anand Srivastav.
\newblock On the query complexity of {B}lack-{P}eg {AB}-{M}astermind.
\newblock {\em Games}, 9(1):Paper No. 2, 12 pp., 2018.
\newblock \href {https://doi.org/10.3390/g9010002}
  {\path{doi:10.3390/g9010002}}.

\bibitem{Foata1965}
Dominique Foata.
\newblock {\'E}tude alg\'ebrique de certains probl\`emes d'analyse combinatoire
  et du calcul des probabilit\'es.
\newblock {\em Publ. Inst. Statist. Univ. Paris}, 14:81--241, 1965.

\bibitem{FoSc}
Dominique Foata and Marcel-P. Sch\"utzenberger.
\newblock {\em Th\'eorie g\'eom\'etrique des polyn\^omes eul\'eriens}, volume
  Vol. 138 of {\em Lecture Notes in Mathematics}.
\newblock Springer-Verlag, Berlin-New York, 1970.
\newblock \href {https://doi.org/10.1007/BFb0060799}
  {\path{doi:10.1007/BFb0060799}}.

\bibitem{garsia:group-actions-o:}
A.~M. Garsia and D.~Stanton.
\newblock Group actions of {S}tanley--{R}eisner rings and invariants of
  permutation groups.
\newblock {\em Adv. in Math.}, 51(2):107--201, 1984.
\newblock \href {https://doi.org/10.1016/0001-8708(84)90005-7}
  {\path{doi:10.1016/0001-8708(84)90005-7}}.

\bibitem{guth:polynomial-meth:}
Larry Guth.
\newblock {\em Polynomial Methods in Combinatorics}, volume~64 of {\em
  University Lecture Series}.
\newblock American Mathematical Society, Providence, RI, 2016.
\newblock URL: \url{https://bookstore.ams.org/ulect-64/}.

\bibitem{haglund:ordered-set-par:}
James Haglund, Brendon Rhoades, and Mark Shimozono.
\newblock Ordered set partitions, generalized coinvariant algebras, and the
  delta conjecture.
\newblock {\em Adv. Math.}, 329:851--915, 2018.
\newblock \href {https://doi.org/10.1016/j.aim.2018.01.028}
  {\path{doi:10.1016/j.aim.2018.01.028}}.

\bibitem{Hiveley2025}
Aurora Hiveley.
\newblock Experimenting with {P}ermutation {W}ordle, 2025.
\newblock URL: \url{https://arxiv.org/abs/2506.23452}, \href
  {https://arxiv.org/abs/2506.23452} {\path{arXiv:2506.23452}}.

\bibitem{Knuth1976}
Donald~E. Knuth.
\newblock The computer as master mind.
\newblock {\em J. Recreat. Math.}, 9(1):1--6, 1976/77.

\bibitem{arxiv-kutin-smithline-multi}
Samuel~A. Kutin and Lawren~M. Smithline.
\newblock Multicolor {Permutation Wordle}.
\newblock in preparation.

\bibitem{arxiv-kutin-smithline}
Samuel~A. Kutin and Lawren~M. Smithline.
\newblock Permutation wordle, 2026.
\newblock URL: \url{https://arxiv.org/abs/2408.00903}, \href
  {https://arxiv.org/abs/2408.00903} {\path{arXiv:2408.00903}}.

\bibitem{LiZhu2024}
Renyuan Li and Shenglong Zhu.
\newblock Playing {M}astermind with {W}ordle-like feedback.
\newblock {\em Amer. Math. Monthly}, 131(5):390--399, 2024.
\newblock \href {https://doi.org/10.1080/00029890.2024.2308489}
  {\path{doi:10.1080/00029890.2024.2308489}}.

\bibitem{martinsson-su:mastermind}
Anders Martinsson and Pascal Su.
\newblock Mastermind with a linear number of queries.
\newblock {\em Combin. Probab. Comput.}, 33(2):143--156, 2024.
\newblock \href {https://doi.org/10.1017/s0963548323000366}
  {\path{doi:10.1017/s0963548323000366}}.

\bibitem{matouv-sek:thirty-three-mi:}
Ji{\v{r}}{\'\i} Matou{\v{s}}ek.
\newblock {\em Thirty-three Miniatures: Mathematical and Algorithmic
  Applications of Linear Algebra}, volume~53 of {\em Student Mathematical
  Library}.
\newblock American Mathematical Society, Providence, RI, 2010.
\newblock URL: \url{https://bookstore.ams.org/stml-53}.

\bibitem{Monkovic2026WordleHard}
Toni Monkovic, Eve Washington, and Tom Giratikanon.
\newblock Wordle's hard mode is actually easier, 730 million games show.
\newblock {\em The New York Times}, 2026.
\newblock URL:
  \url{https://www.nytimes.com/2026/06/18/upshot/wordle-hard-mode.html}.

\bibitem{petersen}
T.~Kyle Petersen.
\newblock {\em Eulerian Numbers}.
\newblock Birkh\"auser Advanced Texts: Basler Lehrb\"ucher.
  Birkh\"auser/Springer, New York, NY, 2015.
\newblock \href {https://doi.org/10.1007/978-1-4939-3091-3}
  {\path{doi:10.1007/978-1-4939-3091-3}}.

\bibitem{math-curse}
Jon Scieszka and Lane Smith.
\newblock {\em Math Curse}.
\newblock Viking, New York, NY, 1995.

\bibitem{Stanley2025}
Richard Stanley.
\newblock Guessing a permutation in a way analogous to {W}ordle.
\newblock MathOverflow, 2025.
\newblock URL: \url{https://mathoverflow.net/q/497111}.

\bibitem{steingrimsson}
Einar Steingr\'imsson.
\newblock Permutation statistics of indexed permutations.
\newblock {\em European J. Combin.}, 15(2):187--205, 1994.
\newblock \href {https://doi.org/10.1006/eujc.1994.1021}
  {\path{doi:10.1006/eujc.1994.1021}}.

\bibitem{Tracy2022Wordle}
Marc Tracy.
\newblock The {New York Times} buys {W}ordle.
\newblock {\em The New York Times}, 2022.
\newblock URL:
  \url{https://www.nytimes.com/2022/01/31/business/media/new-york-times-wordle.html}.

\end{thebibliography}

\end{document}